\documentclass[12pt,notitlepage]{report}
\title{Jordan domains with a rectifiable arc in their boundary}
\author{V. Liontou and V. Nestoridis}

\date{}

\usepackage[english]{babel}
\usepackage[utf8]{inputenc}
\usepackage{titlesec}
\usepackage{amsthm}
\usepackage{amsmath}
\usepackage{amsfonts}
\usepackage{mathrsfs}
\usepackage{enumerate}
\usepackage{tikz}

\titleformat*{\section}{\normalsize\bfseries}
\titleformat*{\subsection}{\Large\bfseries}

\titleformat*{\section}{\normalsize\bfseries}
\titleformat*{\subsection}{\Large\bfseries}
\newtheorem{theorem}{Theorem}[chapter]
\numberwithin{theorem}{section}
\newtheorem{defi}[theorem]{Definition}
\newtheorem{lemma}[theorem]{Lemma}
 \newtheorem{corollary}{Corollary}[theorem]
\newtheorem{prop}[theorem]{Proposition}

\theoremstyle{definition}

\addto\captionsenglish{}
\newcommand\tab[1][1cm]{\hspace*{#1}}

\begin{document}
\maketitle
To the memory of Professor Alain Dufresnoy.
\begin{abstract}
We show that if an open arc J of the boundary of a Jordan domain $\Omega$ is rectifiable, then the derivative $\Phi^{\prime}$ of the Riemann map $\Phi: D\rightarrow \Omega$ from the open unit disk D onto $\Omega$ behaves as an $H^1$ function when we approach the arc $\Phi^{-1}(J^{\prime})$,where $J^{\prime}$ is any compact  subarc of $J$. \\
\break
AMS Classification number:30H10.
\end{abstract}
\section{Introduction}
\tab In \cite{[6]} the Reflection principle has been used in order to prove that if a conformal collar, bounded by a Jordan arc $\delta$ has some nice properties, then any other conformal collar of $\delta$ on the same side has the same nice properties. We use the same method in order to generalize a well-known theorem about rectifiable Jordan curves, \cite{3}.\\
\begin{theorem}
Let $\tau$ be a Jordan curve and $\Phi: D\rightarrow \Omega$ be a Riemann map from the open unit disc D onto the interior $\Omega$ of $\tau$. Then 1 and 2 below are equivalent:\\
1) $\tau$ is rectifiable.\\
2)The derivative $\Phi^{\prime}$ belongs to the Hardy class $H^1$.\\
\end{theorem}

The generalization we obtain is that if $\tau$ is not rectifiable, but an open arc J of it has finite length, then the derivative $\Phi^{\prime}$ behaves as an $H^1$ function when we approach the compact subsets of the arc $\Phi^{-1}(J)\subset \{z\in\mathbb{C}:|z|=1\}$. In the proof we combine the statement of Theorem 1.1 with the Reflection principle, \cite{1}. \\
\tab The above suggests that the Hardy spaces $H^p$ on the disc can be generalized to larger spaces containing exactly all holomorphic  functions $f$ on the open unit disc D, such that $sup_{0<r<1}\int_{a}^{b}|f(re^{it})|^pdt< + \infty$ for some fixed a,b with $a<b<a+2\pi$. One can investigate what is the natural topology on that new space, if it is complete and Baire's theorem can be applied to yield some generic results as non-extendability results, and study properties of the functions, belonging to these spaces. What can be said for their zeros? All these will be investigated in future papers.

\section{Preleminaries}
\tab In order to state our main result we will need some already known results and the lemma 2.2 below.
\begin{defi}
Let $0<p<\infty$. A function $f(z)$ analytic in the unit disk $|z|<1$ is said to be of class $H^{p}$ if $$\frac{1}{2\pi} \int_{0}^{2\pi} |f(r e^{i\theta})|^{p}d\theta$$ remains bounded as $ r \rightarrow 1$.
\end{defi}
\break
The functions of the $H^p$ class share some useful properties such as: 
\begin{enumerate}[(a)]
\item If $U$ is a Jordan domain with rectifiable boundary and $\Phi: D \rightarrow U $ is a Riemann map, then $\Phi ^{\prime} \in H^{1}(D)$.
\item Let $f \in H^{p}$ ; then 
$ \int_{0}^{2\pi}log|f(re^{i\theta})|d\theta \geq log|f(0)|$ and \\$\int_{0}^{2\pi} log |f(re^{i\theta})|d\theta > - \infty $, provided that $f \neq 0$.
\item Let $f \in H^{p} $. Then $f(re^{i\theta})$ has non-tangential limits almost everywhere, on the unit circle, as $r\rightarrow 1^-$.
\end{enumerate}
\begin{lemma}
Let $\gamma $ be a Jordan curve and $J \subseteq \gamma$ a rectifiable, open arc and $J^{\prime} \subseteq J$  a compact arc. Then, $J^{\prime}$ can be extended to a rectifiable Jordan curve $\gamma^{\prime}$ and the interior of $\gamma^{\prime}$ is a subset of the interior of $\gamma$.
\begin{proof}
Let $I$ be a closed interval such that $\gamma(I) $ is a Jordan curve. Let (A,B) be an open interval such that $J:= \gamma ((A,B))$ and let $[a,b]$ be a compact subset of (A,B) such that $J^{\prime}= \gamma([a,b])$. There exists a $t_1$ in (A,B) and $\delta > 0$ such that $A<t_1-\delta<t_1+\delta<a$; thus , $\gamma([t_1-\delta, t_1+\delta])\cap J^{\prime}=\emptyset$ and $$\{\gamma(t_1) \}\cap \gamma(I/(t_1-\delta,t_1+\delta))=\emptyset.$$
Therefore, there exists $\eta>0$ such that $dist(\gamma(t_1), \gamma(I/(t_1-\delta,t_1+\delta)))=\eta>0$ since $I/(t_1-\delta,t_1+\delta)$ is compact and $\gamma$ is continuous.\\
 \tab 
 From the Jordan theorem there exists a sequence $(z_n)_{n\in \mathbb{N}}$ in the interior of the Jordan curve $\gamma$ such that $z_n\rightarrow \gamma(t_1)$. Therefore, there exists a $z_0$ in the interior of $\gamma$ and in the disc $B(\gamma(t_1),\eta/100)$ with center $\gamma(t_1)$ and radius $\eta/100$ and there also exists a $t_1^{\prime}$ in I $:
|z_0-\gamma(t_1^{\prime})|= min (dist(z_0,\gamma(I)))$.\\
\tab We claim that, $$\gamma(t_1^{\prime})\in \gamma([t_1-\delta,t_1+\delta]).$$Let us suppose that $\gamma(t_1)\notin\gamma([t_1-\delta,t_1+\delta])$ to arrive to a contradiction. Then we have 
$|z_0-\gamma(t_1^{\prime})|<\eta/100$ and $|\gamma(t_1)-z_0|<\eta/100$.
Therefore, $$|\gamma(t_1^{\prime})-\gamma(t_1)|<2\eta/100<\eta$$ which contradicts the fact that $dist(\gamma(t_1), \gamma(I/(t_1-\delta,t_1+\delta)))=\eta>0$.
Thus, $t_1^{\prime}\in [t_1-\delta,t_1+\delta]$ and $[z_0, \gamma(t_1^{\prime})]\cap \gamma(I)=\{\gamma(t_1^{\prime})\}$. \\
\tab Therefore, there exists an open segment inside the interior of $\gamma$, which joins $z_0$ with $\gamma(t_1^{\prime})$. We repeat the procedure for $b<t_1-\delta<t_1+\delta<B$ and will find $\gamma(t_2^{\prime})$ and $z_1$ in the interior  of G of $\gamma$, such that the open segment $(z_1, \gamma(t_2^{\prime}))$ is included in G.Therefore, there exists a polygonal line $W$ that connects $z_1$ and $z_0$ in G. It can easily be proven that this polygonal line can be  chosen to be simple. The Jordan curve $$\gamma[t_1^{\prime}, t_2^{\prime}]\cup [\gamma(t_1^{\prime}), z_0]\cup W\cup [z_1, \gamma(t_2^{\prime})]$$ has the desired properties.
This completes the proof of the lemma.

\end{proof}

\end{lemma}

\section{Main Result}
\tab According to a well known theorem of  Osgood - Caratheodory, \cite{5}, every Riemann map, from the open unit disc to the interior of the Jordan curve, extends to a homeomorphism between the closed unit disc and the closure of the Jordan domain. Our main result is the following.
\begin{theorem}
\tab Let $\gamma$ be a Jordan curve and $J\subseteq\gamma$ a rectifiable open arc. Let also $J^{\prime}\subseteq J$ be a compact arc. Let  G be the interior of $\gamma$. Let $\Phi:D\rightarrow G$ be a conformal mapping from the open unit disk D onto G and let $J^{\prime}= \{ \Phi(e^{it}):a\leq t \leq b\}$. Then  $$\int_{a}^{b}|\Phi^{\prime}(r_1e^{it})-\Phi^{\prime}(r_2e^{it})|dt\rightarrow0,$$ as  $r_1,r_2 \rightarrow 1^-$.
\begin{proof}
According to Lemma 2.2 the compact arc $J^{\prime}$ can be extended to a rectifiable Jordan curve $\gamma^{\prime}$ defining a Jordan domain $G^{\prime}\subset G$.\\ \tab Let $f:D\rightarrow G^{\prime}$ be a Riemann map. Thus, $f^{\prime}$ is of class $H^1$ on D. We consider the function $h:D\rightarrow D$, where $h=\Phi^{-1}\circ f$ maps the arc $f^{-1}(J^{\prime})\subseteq \mathbb{T}$ onto the arc $\{e^{it}: a\leq t \leq b\}$, where $\mathbb{T}$ is the unit circle. According to the Reflection Principle the function h is injective and holomorphic on a compact neighbourhood $\overline{V}$ of the compact arc $f^{-1}(J^{\prime})$. Therefore, on $\overline{V}$ the derivative $h^{\prime}$ satisfy $0<\delta<|h^{\prime}(z)|<M< +\infty$ and $h$ (and all its derivatives) are uniformly continuous. We have $\Phi =f\circ h^{-1}= f\circ g$, where $g=h^{-1}$ maps a compact neighbourhood $\overline{W}$ of $\{e^{it}:a\leq t \leq b\}$ biholomorphically on $\overline{V}$ and $0<\overline{\delta}<
 |g^{\prime}(z)|<\overline{M}< + \infty $ on $\overline{W}$ and g (as well as all its derivatives ) are uniformly continuous.
Therefore, $\Phi^{\prime}=f^{\prime}\circ g \cdot g^{\prime}$.\\
\tab There exists $r_0<1$ so that for every $t\in [a,b]$ and every $r\in [r_0,1]$ it holds $re^{it}\in \overline{W}$.
Let $r_1,r_2 \in [r_0,1)$. Then $$|\Phi ^{\prime}(r_1 e^{it})- \Phi ^{\prime}(r_2e^{it})|=$$ $$= |f^{\prime}(g(r_1e^{it}))\cdot g^{\prime}(r_1e^{it})- f^{\prime}(g(r_2e^{it}))g^{\prime}(r_2e^{it})|$$ $$=|f^{\prime}(g(r_1e^{it}))g^{\prime}(r_1e^{it})- f^{\prime}(g(r_2e^{it}))\cdot g^{\prime}(r_1e^{it})$$ $$+f^{\prime}(g(r_2e^{it}))\cdot g^{\prime}(r_1e^{it})-f^{\prime}(g(r_2e^{it}))g^{\prime}(r_2e^{it}) |$$ $$\leq |f^{\prime}(g(r_1e^{it}))-f^{\prime}(g(r_2e^{it}))||g^{\prime}(r_1e^{it})|$$ $$+ |f^{\prime}(g(r_2e^{it}))||g^{\prime}(r_1e^{it})- g^{\prime}(r_2e^{it})|.$$
\break
We also have $|g^{\prime}(r_1e^{it})|\leq \overline{M}$ and $|g^{\prime}(r_1e^{it})- g^{\prime}(r_2e^{it})|\leq \epsilon $ provided that $r_1,r_2 \in [\overline{r_0},1)$, where $\overline{r_0}=\overline{r_0}(\epsilon)\in [r_0,1)$ is given by the uniform continuity of $g^{\prime}$ on $\overline{W}$.
It follows that $$
 \int_{a}^{b}|\Phi^{\prime}(r_1e^{it})-\Phi^{\prime}(r_2e^{it})|dt \leq $$ $$\leq M\int_{a}^{b}|f^{\prime}(g(r_1e^{it}))-f^{\prime}(g(r_2e^{it}))|dt + \epsilon \int_{a}^{b}|f^{\prime}(g(r_2e^{it}))|dt $$
 
 It suffices to show that 
 $$ I_{(r_1,r_2)}= \int_{a}^{b}|f^{\prime}(g(r_1e^{it}))-f^{\prime}(g(r_2e^{it}))|dt $$ is  close to 0 provided that $r_1,r_2$ are sufficiently close to 1 and that $\int_{a}^{b}|f^{\prime}(g(r_2e^{it}))|dt$ stays bounded as $r\rightarrow 1^{-}$.\\
 
 Since g is continuous on $\overline{W}\supset\{e^{it}:a\leq t\leq b \}$ and $f^{\prime}$ has almost everywhere non-tangential limits, if we show that for r close enough to 1 ($r<1$) the complex number $g(re^{it})$ belongs to the angle  $\Gamma_{t,\pi/2}$ with vertex $g(e^{it})$ symmetric with respect to $[0,g(e^{it})]$ with opening $\pi/2$, then we obtain that \(\lim_{r_1,r_2\rightarrow 1^-}|f^{\prime}(g(r_1e^{it}))-f^{\prime}(g(r_2e^{it}))|=0\) almost for every t in $[a,b]$.\\
 \tab Suppose for the moment that we have proven the claim that there exists $\delta\in [\overline{r_0},1)$ so that for all $r\in [\delta,1)$ and all $t\in [a,b]$ we have $g(re^{it})\in \Gamma_{t,\pi/2}$. Then in order to prove that $\lim_{r_1,r_2\rightarrow 1^-}I=0$ we will apply the Dominated Convergence theorem.\\
 \tab Let u denote the non-tangential maximal function $$u(t)=sup\{|f^{\prime}(z)|: z\in \Gamma_{t,\pi/2}, |g(e^{it})-z|<1/2\}.$$ Since $f^{\prime}$ belongs to the Hardy class $H^1$, according to \cite{2}, it follows that u is integrable on $[a,a+2\pi]\supset [a,b]$, We also have $$|f^{\prime}(g(r_1e^{it}))-f^{\prime}(g(r_2e^{it}))|\leq 2u(t).$$ Therefore, $\lim_{r_1,r_2\rightarrow 1^-}I_(r_1,r_2)=0$ and $$\int_{a}^{b}|f^{\prime}(g(r_2e^{it}))|dt\leq \int_{a}^{b}u(t)dt\leq \int_{0}^{2\pi}u(t)dt\leq +\infty $$for all $r_2<1$ close enough to $1$\\
 \tab Now we prove the claim. 
 We have $g(e^{i\theta})=e^{iw(\theta)}, w(\theta)\in \mathbb{R}$. In order to prove that $g(re^{i\theta})\in \Gamma_{\theta,\pi/2}$ it suffices to prove that $|Arg[1-\dfrac{g(re^{i\theta})}{g(e^{i\theta})}]|< \pi/4$. But $$1-\dfrac{g(re^{i\theta})}{g(e^{i\theta})}= \dfrac{g(e^{i\theta})- g(re^{i\theta})}{g(e^{i\theta})}= \int_{[re^{i\theta},e^{i\theta})}^{}\dfrac{g^{\prime}(y)}{g(e^{i\theta})}dy= \int_{r}^{1}\dfrac{g^{\prime}(te^{i\theta})e^{i\theta}}{g(e^{i\theta})}dt$$
Since $g(e^{i\theta})=e^{iw(\theta)}, w(\theta)\in \mathbb{R}$ it follows that $$\dfrac{d}{d\theta}g(e^{i\theta})= g^{\prime}(e^{i\theta})ie^{i\theta}=e^{iw(\theta)}iw^{\prime}(\theta)= g(e^{i\theta})iw^{\prime}(\theta).$$ Thus, $$\dfrac{g^{\prime}(e^{i\theta})e^{i\theta}}{g(e^{i\theta})}= w^{\prime}(\theta)\in \mathbb{R}-\{0\}$$
\tab By continuity of $w^{\prime}$ with respect to $\theta$, we have $w^{\prime}(\theta)\in [c,k]$, for every $\theta\in [a,b]$ or $w^{\prime}(\theta)\in [-k,-c]$ for every $\theta\in [a,b]$, where $0<c<k<+\infty$. The later case is excluded because of the following reason: the function g is a conformal equivalence between two Jordan domains $G^{\prime}$ and $G^{\prime\prime}$ included in D and the boundary of $G^{\prime}$ contains the arc $\{e^{i\theta}:t\in [a,b]\}$ and $g(e^{i\theta})=e^{iw(\theta)}, w(\theta)\in \mathbb{R}$ for all $\theta\in [a,b] $. Let $z_0\in G^{\prime}$; then $g(z_0)\in G^{\prime\prime}\subset D$ and according to the argument principle $Ind(g_{|_{\partial G^{\prime}}}, g(z_0))=1$. If $w^{\prime}(\theta)<0$ then, the homeomorphism $g_{|_{\partial G^{\prime}}}:\partial G^{\prime}\rightarrow \partial G^{\prime\prime}$ turns in such a sense so we should have $Ind(g_{|_{\partial   
 G^{\prime}}}, g(z_0))=-1\neq 1$ impossible. Therefore, $w^{\prime}(\theta)\in [c,k]$ for every $\theta\in [a,b]$ with $0<c<k<+\infty$.
 Thus,$$Arg[1-\dfrac{g(re^{i\theta})}{g(e^{i\theta})}]= Arg \int_{r}^{1}\dfrac{g^{\prime}(te^{i\theta})e^{i\theta}}{g(e^{i\theta})}dt$$
 
 $$= Arg \dfrac{1}{1-r}\int_{r}^{1}\dfrac{g^{\prime}(te^{i\theta})e^{i\theta}}{g(e^{i\theta})}dt$$
 But $lim_{r\rightarrow 1^-}\dfrac{g^{\prime}(re^{i\theta})e^{i\theta}}{g(e^{i\theta})}= \dfrac{g^{\prime}(e^{i\theta})e^{i\theta}}{g(e^{i\theta})}= w^{\prime}(\theta)\in [c,k]$
 for $0<c<k<+ \infty$ and the limit is uniform for $\theta \in [a,b]$. Thus, there exists $\delta\in [\overline{r_0},1)$ so that for every $r\in [\delta,1)$ the quantity $\dfrac{g^{\prime}(re^{i\theta})e^{i\theta}}{g(e^{i\theta})}$ belongs to the convex angle $\{x+iy: 0<x, |y|\leq x\}$ which has vertex 0 and opening $\pi/2$ and is symmetric to the positive x-axis. Its average $\dfrac{1}{1-r}\int_{r}^{1}\dfrac{g^{\prime}(re^{i\theta})e^{i\theta}}{g(e^{i\theta})}dr$ will belong to the same convex angle; therefore, 
 $$|Arg[1-\dfrac{g(re^{i\theta})}{g(e^{i\theta})}]|= |Arg\dfrac{1}{1-r}\int_{r}^{1}\dfrac{g^{\prime}(re^{i\theta})e^{i\theta}}{g(e^{i\theta})}dr|<\pi/4$$
 and the claim is verified. This completes the proof.

\end{proof}
\end{theorem}

\begin{corollary}
For the conformal mapping $\Phi: D\rightarrow G $ in the theorem 3.1 it holds that:
\begin{enumerate}
\item $\int_{a}^{b}|\Phi^{\prime}(re^{it})|dt$ is bounded for $0<r<1$.
\item $\Phi^{\prime}$ has non-tangential limits almost everywhere on $\{e^{it}:a<t<b\}$ which are denoted as $\Phi^{\prime}(e^{it})$ and $\Phi^{\prime}(e^{it})\neq 0$ almost everywhere.
\item $\Phi^{\prime}(e^{it})|_{(a,b)}$ is integrable and $\int_{a}^{b}|\Phi^{\prime}(re^{it})- \Phi^{\prime}(e^{it})|dt\rightarrow 0$ as $r\rightarrow 1^-$.
\item Length of $J^{\prime}= \int_{a}^{b}|\Phi^{\prime}(e^{it})|dt= \lim_{r\rightarrow 1 }\int_{a}^{b}|\Phi^{\prime}(re^{it})|dt=\lim_{r\rightarrow 1^-}$ length of $\Phi\{re^{iu}: a\leq u \leq b\} $.

\end{enumerate}

\end{corollary}
\begin{proof}
\begin{enumerate}
\item \tab From Theorem 3.1, the family $t\rightarrow \Phi^{\prime}(re^{it})$ is Cauchy $L^{\prime}(a,b)$, as $r\rightarrow 1^{-}$. Therefore, there exists the limit $g$ in $L^{\prime}(a,b)$ such that 
$$\int_{a}^{b}|\Phi^{\prime}(re^{it})- g(e^{it})|dt\rightarrow 0.$$
We have $$\int_{a}^{b}|\Phi^{\prime}(re^{it})-g(e^{it})|dt \geq|\int_{a}^{b}|\Phi^{\prime}(re^{it})|dt - \int_{a}^{b}|g(e^{it})|dt|.$$

Therefore, for every $\epsilon>0$ there exists a $r_0>0$, such that for every $r>r_0$ it holds that $$|\int_{a}^{b}|\Phi^{\prime}(re^{it})|dt- \int_{a}^{b}|g(e^{it})|dt|< \epsilon .$$
Since $\int_{a}^{b}|g(e^{it})|dt< +\infty$, it follows that $\int_{a}^{b}|\Phi^{\prime}(re^{it})|dt$ is bounded as $r\rightarrow 1$. 
This completes the proof of 1.

\item \tab We use the notation of Theorem 1. Then $\Phi^{\prime}=f^{\prime}(h)h^{\prime}$, where $h:D \rightarrow D $. Since $f^{\prime}\in H^1(D)$ there exists the non-tangetial limit a.e. on $\partial D$ and therefore on $J^{\prime}$.\\
\tab On the other hand, the function h is holomorphic on D and can be extended holomorphically on a neighbourhood of $J^{\prime}$. Therefore, h and $h^{\prime}$ have non-tangetial limits a.e. on $\{e^{i\theta}:a<\theta<b\}$. Thus, $\Phi^{\prime}= f^{\prime}(h)h^{\prime}$ has non-tangetial limits a.e. on $\{e^{i\theta}, a<\theta<b\}$.\\
\tab Now, $f^{\prime}$ is in $H^{1}$ and $f^{\prime}\neq 0$. Thus, $f^{\prime}(h(e^{i\theta}))\neq 0$ a.e. on (a,b). Also $h^{\prime}(e^{it})\neq 0$ for all $t\in (a,b)$ because h is injective and holomorphic on a compact neighbourhood of $J^{\prime}$. Thus, $\Phi ^{\prime}(e^{it})\neq 0$ almost everywhere on (a,b). This completes the proof of 2.
\item Since $\Phi^{\prime}(re^{it})$ is Cauchy in $L_1$ as $r\rightarrow 1^-$, there exists $g(e^{it}):= \lim _{r\rightarrow1}\Phi^{\prime}(re^{it})$ in $L_1$. There exists a sequence $r_{k_n}$, \cite{7}, such that $\Phi^{\prime}(r_{k_n}e^{it})\rightarrow g(e^{it})$ a.e. But $\Phi ^{\prime}(re^{it})\rightarrow \Phi^{\prime}(e^{it})$ a.e. on $\{e^{it}, a<t<b\}$ non-tangetially. Therefore $g=\Phi^{\prime}(e^{it})$ a.e.
Since $g\in L_1$ and $g=\Phi^{\prime}(e^{it})$ a.e., it follows that $\Phi^{\prime}\in L_1$.
This completes the proof of 3.

\item Let A, B be such that $J^{\prime}=\{f(e^{it}): A\leq t \leq B\}$. Since $f^{\prime}\in H^1(D)$ we have length of $J^{\prime}= \int_{A}^{B}|f^{\prime}(e^{it})|dt$, \cite{3}. But $f=\Phi\circ h$, therefore $f^{\prime}=\Phi^{\prime}\circ h\cdot h^{\prime}$. Thus, length $J^{\prime}=\int_{A}^{B}|\Phi^{\prime}(h(e^{it}))|h^{\prime}(e^{it})|dt$.We do the diffeomorphic change of variable $h(e^{it})= e^{iu}$ that is $$e^{it}=h^{-1}(e^{iu}) $$ which implies $$ie^{it}dt={(h^{-1})}^{\prime}(ie^{iu})\cdot ie^u du $$ and $$dt=|(h^{-1})^{\prime}(e^{iu})|du=\dfrac{1}{|h^{\prime}(e^{it})|}du.$$ \tab According to \cite[2.6,pg. 74]{4}, for this change of variable for integrable functions we find length $J^{\prime}=\int_{a}^{b}|\Phi^{\prime}(e^{iu})|du$.\\
Using part 3, we take 
$$\int_{a}^{b}|\Phi^{\prime}(e^{iu})|du=$$ $$=\lim_{r\rightarrow 1^-}\int_{a}^{b}|\Phi^{\prime}(re^{iu})|du=$$ $$=\lim_{r\rightarrow 1^-}length\{\Phi(re^{iu}): a\leq u \leq b\}.$$
The result easily follows. This completes the proof of part 4 and of the whole Corollary.\\
\break
However, we will give a second alternative proof for part 4.\\
\tab Since $\Phi(e^{it})$ is of bounded variation on $[a,b]$, the arc measure on $J^{\prime}$ is $|\Phi^{\prime}(e^{iu})|du+dv$, where $dv$ is a singular non negative measure; it follows that
$$length  J^{\prime}\geq \int_{a}^{b}|\Phi^{\prime}(e^{iu})|du, \cite{4}. $$
We notice that, combining the relation $\Phi^{\prime}=f^{\prime}\circ h \cdot h^{\prime}$ with the fact that $f^{\prime}\in H^{1}$, we easily conclude that the non-tangential limits of $\Phi^{\prime}$ on $\{ e^{iu}:a\leq u \leq b\}$ coincide almost everywhere with the derivative $\dfrac{d\Phi}{de^{iu}}(e^{iu})$ computed for the restriction of $\Phi$ on $\{e^{iu} :a\leq u\leq b\}$, which exists almost everywhere on $\{e^{iu}:a\leq u \leq b\}$ , because $J^{\prime}$ is rectifiable and $\Phi(e^{iu})$ is of bounded variation on $[a,b]$
According to part 3, we have $$\int_{a}^{b}|\Phi^{\prime}(e^{iu})|du= $$ $$=\lim_{r\rightarrow 1^-}\int_{a}^{b}|\Phi^{\prime}(re^{iu})|du=$$ $$=\lim_{r\rightarrow1^-}length \{\Phi(re^{iu}): a\leq u\leq b\}$$

Since $\Phi(re^{iu})\rightarrow \Phi(e^{iu})$ as $r\rightarrow1^-$ we have $$ lengthJ^{\prime}=length \{ \Phi(e^{iu}): a\leq u \leq b\}\leq$$ $$\leq \liminf_{r\rightarrow1^-}length\{\Phi(e^{iu}: a\leq u \leq b)\}$$  
(see Prop.  4.1 below). Now the result easily follows. The proof is complete.

\end{enumerate}

\end{proof}
\section{Further results }
We have seen that $\lim_{r\rightarrow 1^{-}}$ length $\Phi\{re^{it}:a\leq t \leq b\}=$ length of $\Phi\{e^{it}:a\leq t \leq b\}$ provided that for some $a^{\prime}, b^{\prime}:a^{\prime}<a<b<b^{\prime}$ the length of $\Phi\{e^{it}:a^{\prime}\leq t \leq b ^{\prime}\}$ is finite. Composing $\Phi$ with an automorphism of the open unit disc $w(z)=c\dfrac{z-\gamma}{1-\overline{\gamma}z}, |c|=1, |\gamma|<1$ we can obtain similar results of other families of curves converging to $\Phi\{e^{it}:a\leq t\leq b\}$. We will not insist towards this direction.
For any arc $\{e^{it}:A\leq t\leq B\}, A<B< A+ 2\pi$, we have the following:
\begin{prop}
Under the above assumptions and notation we have   the following inequality.\\
length of $\Phi\{e^{it}: A \leq t \leq B\}\leq \liminf_{r\rightarrow 1^{-}}$ of length $\Phi \{re^{it}: A \leq t \leq B\}$. 
\end{prop}
\begin{proof}
\tab Let $r_n<1, r_n\rightarrow 1$ and M be such that $$length\Phi\{r_n e^{it}: A\leq t \leq B\}\leq M$$ for all n. Then we will show that length of $\Phi\{e^{it}: A\leq t \leq B\}\leq M$. It suffices to prove that 
$$\sum_{y=0}^{N-1}|\Phi(e^{it_{y+1}})- \Phi(e^{it_{y}})|\leq M $$ for any partition $t_0= A< t_1 <....< t_{N-1}< t_N=B$.\\
But $$\sum_{0}^{N-1}|\Phi(r_n e^{it_{y+1}})- \Phi(r_n e^{it_{y}})|\leq $$ $$length \Phi\{r_n e^{it}: A \leq t \leq B\}\leq M.$$ Since $\Phi(r_n e^{it})\rightarrow \Phi(e^{it}), n\rightarrow +\infty$, passing to the limit we obtain $\sum_{0}^{N-1}|\Phi(e^{it_{y+1}})- \Phi (e^{it_{y}})|\leq M$. The result easily follows. 

\end{proof}
\begin{corollary} 
Under the above assumptions and notations we have the following:
\begin{enumerate}
\item \tab If $length\Phi(\{e^{it}: A\leq t \leq B\})= + \infty$, then $$length\Phi(\{e^{it}: A \leq t \leq B\})= \lim_{r\rightarrow 1^{-}} length\Phi(\{re^{it}: A\leq t \leq B\})$$ .

\item \tab If there exists $A^{\prime}, B^{\prime}, A^{\prime}<A<B < B^{\prime}$ such that  $length\Phi\{e^{it}: A^{\prime}\leq t\leq B^{\prime}\}< + \infty$, then  $$length\Phi\{e^{it}: A \leq t \leq B\}= \lim_{r\rightarrow 1^{-}} length \Phi\{re^{it}: A\leq t \leq B\}$$.
\end{enumerate}
\end{corollary}
The proof of the corollary 4.1.1 follows easily from the previous results. \\
We believe that it is possible to have:
 $$length\Phi\{e^{it}:A\leq t \leq B\}< + \infty$$
and 
 $$length\Phi\{e^{it}:A\leq t \leq B\}
< \liminf_{r\rightarrow 1^-}length\Phi\{re^{it}: A\leq t \leq B\}
<$$ $$<\limsup_{r\rightarrow 1^-} length  \Phi\{re^{it}:A\leq t \leq B\}$$
but we do not have an example. A candidate for such an example is the Jordan domain $$\Omega=\{x+iy: -5<y<x \cos(1/x); 0<x<1\}\cup \{x+iy: -5< y< 0, -1<x\leq 0\}.$$\\
\tab Although $\int_{0}^{2\pi}|f^{\prime}(re^{it})|dt$ is increasing with respect to $r\in (0,1)$, we believe that this is no longer true for $\int_{a}^{b}|\Phi^{\prime}(re^{it})|dt$ and a candidate for a counter example is any convex polygonal domain $\Omega$.\\
\break
Finally, we have the following:
\begin{theorem}
Let $\Omega$ be a Jordan domain and $\Phi:D\rightarrow \Omega$ a Riemann map from the open unit disc $D$ onto $\Omega$. Let $A<B<A+2\pi$, then the following are equivalent.
\begin{enumerate}
\item \tab For every $ a,b $ such that  $ A<a<b<B $ the arc $\{\Phi(e^{it}):a\leq t\leq b \}$ is rectifiable.
\item \tab For every $a,b$ such that $A<a<b<B$ we have $$sup_{0<r<1}\int_{a}^{b}|\Phi^{\prime}(re^{it})|dt= M_{a,b}<\infty$$
\item \tab For every $a,b$ such that $A<a<b<B$ there exist curves \\$\gamma_r:[a,b]\rightarrow\mathbb{C}, 0<r<1$ such that $\lim_{r\rightarrow 1^-}\gamma_r(t)=\Phi(e^{it})$ for all $t\in[a,b]$ and such that the lengths of $\gamma_r$ are uniformly bounded as $r\rightarrow 1^-$, by a constant $C_{a,b}<\infty$.
\end{enumerate}
\end{theorem}

\begin{proof}
We have already seen that $ 1.\Rightarrow 2.$
 In order to see that $2.\Rightarrow 3.$ it suffices to set $\gamma_r(t)=\Phi(re^{it})$. Finally, to prove that $3.\Rightarrow 1.$, it suffices to prove that $$\sum_{j=0}^{n-1}|\Phi(e^{it_{j+1}})-\Phi(e^{it_j})|\leq C_{a,b}$$ for all partitions $a=t_0<t_1<...<t_{n-1}<t_n=b$. But
 $$\sum_{j=0}^{n-1}|\gamma_r(t_{j+1}) - \gamma_r(t_j)| \leq length\gamma_r\leq C_{a,b}$$ and $\lim_{r\rightarrow 1^-}\sum_{j=0}^{n-1}|\gamma_r(t_{j+1}) - \gamma_r(t_j)|=\sum_{j=0}^{n-1}|\Phi(e^{it_{j+1}})-\Phi(e^{it_j})|$ and the proof is completed. 
\end{proof}
\textbf{ Acknowledgement}: We would like to thank professor E.Katsoprinakis for his interest in this work.

{1}

\begin{thebibliography}{1}
\bibitem{1}Ahlfors, Complex analysis,Second edition ,McGraw-Hill,New York,1966

\bibitem{2}D.L.Burkholder, R.F.Gundy, M.L.Silverstein, A maximal function,  characterization of the class $H^p$. Transactions of AMS, Volume 157, June 1971,pages  137-153

\bibitem{3} P.L.Duren, Theory of $H^p$ spaces, Academic Press, New York and London 1970

\bibitem{4}G.B. Foland, Real Analysis, Modern techniques and their Applications, 2nd edition, Wileg 1999 

\bibitem{5} Koosis P. ,An introduction to $H_p$ spaces, Cambridge University Press, 1998V.
\bibitem{[6]}Liontou, V. Nestoridis, One sided conformal collars and the reflection principle, arxiv: 1612.00177

\bibitem{7}W. Rudin, Real and Complex Analysis, Mc-Graw- Hill, New York, 1966.

 National and Kapodistrian University of Athens,\\ Department of Mathematics \\15784 \\Panepistemiopolis\\ Athens GREECE
\\
e-mail: lvda20@hotmail.com\\
e-mail: vnestor@math.uoa.gr


\end{thebibliography}
\end{document}